\documentclass[12pt]{amsart}

\newtheorem{theorem}{Theorem}
\newtheorem{lemma}{Lemma}
\newtheorem{conjecture}{Conjecture}
\newtheorem{proposition}{Proposition}
\newtheoremstyle{rem}{0.4cm}{0.4cm}{}{}{\itshape}{.}{.5em}{}
\theoremstyle{rem}
\newtheorem{remark}{Remark}

\newcommand{\Z}{\mathbb{Z}}
\newcommand{\Q}{\mathbb{Q}}
\newcommand{\R}{\mathbb{R}}
\newcommand{\C}{\mathbb{C}}
\newcommand{\T}{\mathbb{T}}
\newcommand{\Lc}{\mathcal{L}}

\newcommand{\tr}{\operatorname{tr}}
\newcommand{\Gal}{\operatorname{Gal}}

\newcommand{\ord}{\operatorname{ord}}
\newcommand{\id}{\mathrm{id}}
\newcommand{\new}{\text{new}}

\usepackage[all]{xy}
\usepackage{geometry}
\usepackage{color}
\usepackage{marginnote}
\definecolor{light-gray}{gray}{0.5}

\usepackage[linktocpage,bookmarksdepth=2]{hyperref}
\hypersetup{
     colorlinks,
     linkcolor=green,
     citecolor=green,
     urlcolor=blue
}
\usepackage{geometry}
\usepackage{amssymb}

\author{Sandro Bettin}
\address{
  Dipartimento di Matematica, Universit{\`a} di Genova, via Dodecaneso 35, 16146
  Genova, Italy
}
\email{bettin@dima.unige.it}
\author{Corentin Perret-Gentil}
\address{
  Centre de recherches math{\'e}matiques, Universit{\'e} de Montr{\'e}al, Montr{\'e}al, Canada}
\email{corentin.perretgentil@gmail.com}
\author{Maksym Radziwi\l\l}
\address{
  Caltech, Department of Mathematics, 1200 E California Blvd, Pasadena, CA,
  91125, USA
}
\email{maksym.radziwill@gmail.com}

\title{A note on the dimension of the largest simple Hecke submodule}

\begin{document}
\begin{abstract}
  For $k\ge 2$ even, let $d_{k,N}$ denote the dimension of the largest simple Hecke submodule of $S_{k}(\Gamma_0(N); \Q)^\new$. We show, using a simple analytic method, that $d_{k,N} \gg_k \log\log N / \log(2p)$ with $p$ the smallest prime co-prime to $N$. Previously, bounds of this quality were only known for $N$ in certain subsets of the primes. We also establish similar (and sometimes stronger) results concerning $S_{k}(\Gamma_0(N), \chi)$, with $k \geq 2$ an integer and $\chi$ an arbitrary nebentypus. 
\end{abstract}

\maketitle

\section{Introduction}

For an integral weight $k\ge 2$ and a level $N\ge 1$, the anemic Hecke $\Q$-algebra
\[\T:=\Q[T_n : (n,N)=1],\]
generated by the Hecke operators $T_{n}$, acts on the space of cusp forms $S_k(\Gamma_0(N))$.

Simple Hecke submodules of $S_k(\Gamma_0(N))$ of dimension $d$ correspond to $\Gal(\overline\Q/\Q)$-orbits of size $d$ of (arithmetically) normalized eigenforms $f\in S_k(\Gamma_0(N))$. When $k=2$, the work of Shimura also gives a correspondence with simple factors of dimension $d$ of the Jacobian $J_0(N)$ of the modular curve $X_{0}(N)$.
Thus it is interesting to ask about the dimension $d_{k,N}$ of the largest simple Hecke submodule of $S_{k}(\Gamma_0(N))$, or equivalently the maximal degree of Hecke fields of normalized eigenforms.

Maeda \cite{HidMae97} postulated that $S_{k}(\Gamma_0(1))$ is a simple Hecke module for all even $k \geq 12$. This deep conjecture implies among other things that $L(\tfrac 12, f) \neq 0$ for all $f \in S_{k}(\Gamma_0(1))$, see \cite{ConreyFarmer}. When $N>1$, there is an obstruction to simplicity due to the Atkin--Lehner involutions, but numerical evidence suggests that this is the only asymptotic barrier when $N$ is square-free. This led Tsaknias \cite{Tsa14} to suggest the following generalization of Maeda's conjecture (see also \cite{DieuTsa16} for non-square-free levels):
\begin{conjecture}\label{conj}
  For $k\ge 2$ even and large enough and $N$ square-free, the number of Galois orbits of newforms in $S_k(\Gamma_0(N))$ is $2^{\omega(N)}$. In particular, for any fixed $\varepsilon > 0$ we have
  \[d_{k,N}\gg_{k,\varepsilon} N^{1-\varepsilon}.\]
That is, there exists a constant $c(k, \varepsilon) > 0$ depending at most on $k$ and $\varepsilon$ such that $d_{k, N} > c(k, \varepsilon) N^{1 - \varepsilon}$ for all square-free $N \geq 1$. 
\end{conjecture}

There is a massive gap between Conjecture \ref{conj} and the unconditional results. Through an equidistribution theorem for Hecke eigenvalues, Serre \cite{Ser97} was the first to establish that $d_{k,N} \rightarrow \infty$ as $k+N \rightarrow \infty$. Subsequently, by making Serre's equidistribution theorem effective, Royer \cite{Roy00} and Murty--Sinha \cite{MurtSinh09} showed that $d_{k,N} \gg_{k,p} \sqrt{\log\log N}$ for any $p\nmid N$. In the particular case where $N$ lies in a restricted set of primes, this bound has been improved by several authors. Extending a method of Mazur to all even weights, Billerey and Menares \cite[Theorem 2]{BilMen16} obtained that $d_{k,N}\gg_k \log{N}$ when $N\ge (k+1)^4$ is in a explicit set primes of lower natural density $\ge 3/4$. When the lower bound is fixed in advance and one looks for a level with a given number of prime divisors attaining it, see also \cite{DieuJimRib15}. When $N \equiv 7 \pmod{8}$ is prime, Lipnowski--Schaeffer \cite[Corollary 1.7]{LipnSchaeff18} also showed that $d_{2,N} \gg \log\log N$, which can be significantly improved for $N$ in certain subsets of the primes under certain well-known conjectures and heuristics.\\

In this paper we show that bounds of Lipnowski--Schaeffer quality can be obtained for all levels and integer weights. Our method is however, analytic and we believe simpler than the one in \cite{LipnSchaeff18}.  

\begin{theorem} \label{thm1}
  Let $k\ge 2$ even and $N\ge 1$ be integers. Then the dimension of the largest simple Hecke submodule of $S_{k}(\Gamma_0(N))^{\new}$ is
  \[d_{k,N} \gg_{k}\frac{\log\log{N}}{\log(2p_N)},\]
  as $N \rightarrow \infty$, where $p_N$ denotes the smallest prime co-prime to $N$.
\end{theorem}

Since the vast majority of integers $N$ have a small co-prime factor, this bound is essentially asserting that $d_{k,N} \gg_{k} \log\log N$.  Theorem \ref{thm1} appears to be the first bound of ``$\log\log N$ strength'' for any even weight $k \geq 4$, and in the case $k = 2$, without restriction on the level.\\

We state below a more general and precise form of Theorem \ref{thm1} that holds in the presence of a nebentypus.
\begin{theorem} \label{thm2}
  Let $k \geq 2$ and $N \geq 1$ be integers. Let $p\nmid N$ and let $\chi: (\Z/N)^\times\to\C^\times$ be a homomorphism such that $\chi(-1)=(-1)^k$. Then the maximum size of the $\Gal(\overline\Q/\Q)$-orbits of newforms $f\in S_k(\Gamma_0(N),\chi)$ is 
  \[\ge \frac{2}{(k-1)\log(4p)}\cdot \log \left( \frac{\log N}{2\pi \log p} \right)\]
  for all sufficiently large $N$ (in terms of $k$).
\end{theorem}
By definition, the same lower bound holds for the maximum degree of the Hecke fields $K_f$ of newforms $f$ (see Section \ref{sec:proof12}). Note that $K_f$ always contains the cyclotomic field $\Q(\zeta_{\ord(\chi)})$ generated by the values of $\chi$ (a consequence of the Hecke relations at $p^2$, see Lemma \ref{lemma:strongMult} below), so the trivial lower bound in both cases is $\varphi(\ord(\chi))$.

\begin{remark}
  The result of Billerey--Menares mentioned above actually shows that when $\ell\ge(k+1)^4$ belongs to an explicit set $\Lc$ of primes with lower density $\ge 3/4$, there exists a normalized eigenform $f\in S_k(\Gamma_0(\ell))^{\new}$ with $\deg(K_f)\gg_k \log{\ell}$. Hence, for $\varepsilon>0$, if an integer $N$ has a prime factor $\ell>N^{\varepsilon}$ that lies in $\Lc$, then $\deg(K_f)\gg_{k,\varepsilon}\log{N}$ for some $f\in S_k(\Gamma_0(\ell))\hookrightarrow S_k(\Gamma_0(N))$. Hence, Theorem \ref{thm2} with ``newform'' replaced by the weaker conclusion ``normalized eigenform'' would follow from \cite[Theorem 2]{BilMen16} for almost all integers $N$.
\end{remark}

In certain special situations it can be shown that the degree of the number field $K_f$ is large for \emph{all} newforms $f \in S_{k}(\Gamma_0(N),\chi)$. For instance, when $p^{r} \mid N$ Brumer \cite[p.3, Theorem 5.5, Remark 5.7]{Brum95} showed that $K_f$ contains the maximal real subfield of the $p^s$-th roots of unity, where $s=\lceil{\frac{r}{2}-1-\frac{1}{p-1}}\rceil$ (see also \cite{MO10,MR2046205}).

We exhibit a similar phenomenon which sometimes allows to significantly improve on Theorem \ref{thm2} and the trivial bound $\deg K_f \geq \varphi(\ord \chi)$, when $k$ is odd, depending on the nebentypus $\chi$ and the factorization of $N$.

  \begin{theorem}\label{thm3}
    Let $k\ge 3$ be an odd integer, $N\ge 1$ be square-free, $\chi: (\Z/N)^\times\to\C^\times$ be a homomorphism such that $\chi(-1)=(-1)^k$, and decompose
    \[N_2=\prod_{\substack{p\mid N\\ \chi_p=1}}p, \hspace{1cm}\chi=\prod_{p\mid N}\chi_p,\hspace{0.5cm} \text{with}\hspace{0.5cm} \chi_p:(\Z/p)^\times\to\C^\times.\]
    Then, for any newform $f\in S_k(\Gamma_0(N),\chi)$,
    \[\deg{K_f}\geq  \varphi(\ord(\chi)) \cdot 2^{\omega(N_2)-\omega((N_2,2\ord(\chi))) - 1} ,\]
    In particular, if $(N_2,2 \ord(\chi))=1$, then
    \[\deg{K_f}\geq \varphi(\ord(\chi)) \cdot 2^{\omega(N_2) - 1}.\]   
  \end{theorem}
  For example, given $\varepsilon > 0$ and $k\ge 3$ odd, for a ``typical'' square-free integer $N$ and $\chi$ a random quadratic character mod $N$ (resp. the trivial character), we get
  \[\deg K_f\gg_{\varepsilon} (\log{N})^{\frac{\log{2}}{2} - \varepsilon} \qquad(\text{resp. }\gg_{\varepsilon} (\log{N})^{\log{2} - \varepsilon})\]
  for all newforms $f\in S_k(\Gamma_0(N),\chi)$. In fact it is possible to extend Theorem \ref{thm3} to the case of non-square-free $N$, but we maintain this restriction to keep the exposition simple.

\subsubsection*{A short outline of the proofs}
  
We will now say a few words about the proof of these theorems and the limitations of our method of proof.

 The proof of Theorem \ref{thm1} and Theorem \ref{thm2} proceeds by observing that if we can find a newform $f$ for which the eigenvalue $a_f(p_N)$ is abnormally small in absolute value but non-zero, then the degree of the corresponding Hecke field $K_f$ needs to be large (see Proposition \ref{lem1}). We then use the equidistribution of Hecke eigenvalues (in the form of Murty--Sinha) to prove the existence of such an $f$. This contrasts with the previous analytic approaches in which one probed (using the equidistribution of Hecke eigenvalues) the neighborhood of every algebraic integer up to a certain height. 

The proof of Theorem \ref{thm3} proceeds by first noticing that by strong multiplicity one, the number field $\Q(a_f(n) : n\ge 1)$ coincides with $K_f = \Q(a_f(n) : (n, N) = 1)$. Subsequently we focus exclusively on the ramified primes $p \mid N$. For $k$ odd, the coefficient of $f$ at $p\mid N_2$ is equal to $\sqrt{p}$ multiplied by a factor lying in a small extension of $K_f$ (the eigenvalue of an Atkin--Lehner operator). Considering all these divisors yields the factor $2^{\omega(N_2)}$.

\subsubsection*{Limitations of the method}

The best result that the method of proof of Theorem \ref{thm1} and Theorem \ref{thm2} can theoretically deliver is for each $k$ even and $N \geq 1$ the existence of an $f \in S_{k}(\Gamma_0(N))$ such that  $\deg K_f \gg_{k} \log N$. To see this consider for simplicity $k$ fixed and $N$ odd. Then we expect that the coefficients $a_f(2)$ with $f$ varying in $S_{k}(\Gamma_0(N))$ behave as a collection of roughly $\asymp_{k} N^{1 + o(1)}$ random numbers distributed according to the Sato-Tate law. Therefore by linearity of expectation for any given $\varepsilon > 0$ we expect that there exists a form $f \in S_{k}(\Gamma_0(N))$ with $0 < |a_f(2)| \ll_{k} N^{-1 + \varepsilon}$ and moreover that this is best possible up to the factor $N^{\varepsilon}$. Plugging this into Proposition \ref{lem1} would result in a lower bound $\deg K_f \gg_{k} \log N$ for some $f \in S_k(\Gamma_0(N))$. Note that the existence of a $\delta > 0$ such that for all $k$ fixed and $N$ odd there exists an $f \in S_{k}(\Gamma_0(N))$ with $0 < |a_f(2)| \ll_{k} N^{-\delta}$ would be also enough to obtain the lower bound $\deg K_f \gg_{k} \log N$ for some $f \in S_{k}(\Gamma_0(N))$.

  \subsection*{Acknowledgments}
The work of the first author is partially by PRIN 2015 ``Number Theory and Arithmetic Geometry''.
The third author would like to acknowledge the support of a Sloan fellowship. We would like to thank  Nicolas Billerey, Armand Brumer, and Ricardo Menares for comments on the manuscript. We would like to thank the referees for a careful reading of the paper and useful suggestions.

\section{Proof of Theorem \ref{thm1} and Theorem \ref{thm2}}\label{sec:proof12}
Throughout let $k \geq 2$ and $N \geq 1$ be integers, and $\chi: (\Z/N)^\times\to\C^\times$ a homomorphism such that $\chi(-1)=(-1)^k$.
Let $f\in S_k(\Gamma_0(N),\chi)$ be a normalized eigenform with Fourier expansion
$$
f(z) := \sum_{n \geq 1} a_f(n) e(n z), \qquad a_f(1)=1, \qquad e(z) := e^{2\pi i z}. 
$$
Given a prime $p\nmid N$, we also define (for reasons that will become clear when proving Lemma \ref{lem2})
\[a_f'(p)=\frac{a_f(p)}{2p^{\frac{k-1}{2}}\sqrt{\chi(p)}}\in\R,\]
for a fixed choice of square root.

Since simple Hecke submodules of $S_k(\Gamma_0(N))$ of dimension $d$ correspond to $\Gal(\overline\Q/\Q)$-orbits of size $d$ of (arithmetically) normalized eigenforms $f\in S_k(\Gamma_0(N))$ (see \cite{DiamondIm95}), it suffices to obtain lower bounds for
\[\max_{f\in S_k(\Gamma_0(N),\chi)}\deg K_f,\qquad K_f=\Q\left(a_f(n) : (n,N)=1\right),\]
where $f$ runs over newforms, to prove Theorems \ref{thm1} and \ref{thm2}.

 The first input to our argument is a simple lemma from diophantine approximation, that allows to pass from small values of $|a_f(p)|$ to lower bounds for the degree of the Hecke field.

 \begin{proposition} \label{lem1}
   If $p\nmid N$ and $a_f(p)\neq 0$, then
   \[\deg \Q(a_f(p))\ge \frac{2}{k-1}\cdot \frac{\log{\frac{1}{|a_f'(p)|}}}{\log(4p)}.\]  
\end{proposition}
\begin{proof}
  Since $a_f(p)$ is an algebraic integer \cite[Corollary 12.4.5]{DiamondIm95}, its norm is a nonzero integer. Thus if we denote by $g$ the degree of $a_f(p)$ and by $a_{f, 1}(p), \ldots, a_{f,g}(p)$ all of the conjugates of $a_f(p)$ (including $a_f(p)$ itself), then, 
  \[\prod_{i=1}^g|a_{f,i}(p)|\ge 1.\]
  By Deligne's proof of the Ramanujan--Petersson conjecture for $f$ \cite{Del71}, $a_f(p)$ is the sum of two $p$-Weil numbers of weight $k-1$, so $|a_{f,i}(p)|\le 2p^{\frac{k-1}{2}}$ for all $i$. Therefore,
  \[\prod_{i=1}^g|a_{f,i}(p)|\le |a_f(p)|\left(2p^{\frac{k-1}{2}}\right)^{g-1}\]
  and the claim follows.
\end{proof}
\begin{remark}
  If $\Gamma_f\le\Gal(K_f/\Q)$ is the group of inner twists of $f$ (see \cite[Section 3]{Ribet80}, \cite[Section 3]{Ribet85}), then the proof of Proposition \ref{lem1} shows that the lower bound can actually be improved by a factor of $|\Gamma_f|$ (or even $|\Gamma_f|^2$ if $\chi(p)\in\Q^\times$). In the case $k=2$, $\chi=1$, $N$ square-free, there are no nontrivial inner twists, but otherwise it is believed that $|\Gamma_f|$ could become large; if $\chi^2\neq 1$, there is always a nontrivial inner twist given by conjugation.
\end{remark}
We will now use the equidistribution of Hecke eigenvalues to exhibit a newform $f\in S_{k}(\Gamma_0(N), \chi)^\new$ for which $a_f(p)$ is abnormally small, yet non-zero. This will therefore give a lower bound for the degree of $\Q(a_f(p))$ and thus a lower bound for the degree of $K_f$.

\begin{lemma} \label{lem2}
  Let $p\nmid N$. There exists a newform $f\in S_k(\Gamma_0(N),\chi)^{\new}$ such that,
  \begin{equation}
    \label{eq:apsmall}
    0 < |a_f'(p)| \leq \frac{\pi}{2} \cdot \frac{p + 1}{p} \cdot \frac{\log p}{\log N}
  \end{equation}
  for all sufficiently large $N$ (in terms of $k$).  
\end{lemma}
\begin{proof}
  Let $B_k(\Gamma_0(N),\chi)$ be the $\overline\Q$-basis of $S_k(\Gamma_0(N),\chi)^{\new}$ composed of the $d_{k,N,\chi}$ newforms at level $N$. For $(n,N)=1$, let us also normalize Hecke operators acting on $S_k(\Gamma_0(N),\chi)^\new$ as $T_n':=T_n/(2n^{\frac{k-1}{2}}\sqrt{\chi(n)})$. By \cite[Sections 5.1,\,5.3]{Ser97}, the normalized eigenvalues $(a_f'(p))_{f\in B_k(\Gamma_0(N),\chi)}$ are distributed in $[-1,1]$ as $N\to\infty$ according to a measure converging to the Sato--Tate measure as $p\to\infty$. 

  For $A\in(0,1)$, let us give a lower bound on
  \begin{eqnarray*}
    C_{k,N,\chi}(A)&:=&\frac{|\{f\in B_k(\Gamma_0(N),\chi) : 0<|a'_f(p)|\le A\}|}{d_{k,N,\chi}}.
  \end{eqnarray*}
  If the nebentypus is trivial and we do not necessarily want to find a form that is new, we can directly apply \cite[Theorem 19]{MurtSinh09} to get \eqref{eq:Clowerbound2} below. In general, \cite[Theorem 8, Lemma 17, Section 10]{MurtSinh09} show that for any $M\ge 1$,
  \begin{eqnarray}
    &&\left|C_{k,N,\chi}(A)-2\int_0^A F(-x)dx\right|\label{eq:ErdosTuran}\\
    &&\le \frac{1}{M+1}+\sum_{1\le|m|\le M} \left(\frac{1}{M+1}+\min \left(2A,\frac{1}{\pi|m|}\right)\right)\left|\frac{\tr\left(T'_{p^{|m|}}-T'_{p^{|m|-2}}\right)}{d_{k,N,\chi}} - c_m\right|\nonumber,
  \end{eqnarray}
  where $c_m=\lim_{k + N \to\infty} \tr(T'_{p^{|m|}}-T'_{p^{|m|-2}})/d_{k,N,\chi}$ and $F(x)=\sum_{m\in\Z} c_me(mx)$, with the convention that $T'_{n}=0$ if $n<1$. The Eichler--Selberg trace formula for $S_k(\Gamma_0(N),\chi)$ \cite[(34)]{Ser97} and \cite[Section 5.3]{Ser97} gives that,
  \begin{eqnarray*}
    \tr T'_{p^m}&=&\sum_{N_1\mid N}d^*(N/N_1)\Big(A_{\text{main}}(k,N_1,T'_{p^m})+A_{\text{ell}}(k,N_1,\chi,T'_{p^m})\\
                &&\hspace{3cm}+A_{\text{hyp}}(k,N_1,\chi,T'_{p^m})+\delta_{\substack{k=2\\\chi=1}} A_{\text{par}}(k,N_1,T'_{p^m})\Big),
  \end{eqnarray*}
  for any $m\ge 1$, with the main, elliptic, hyperbolic and parabolic terms given in \cite[(35,\,39,\,45,\,47)]{Ser97}, and where $d^*$ is the multiplicative function defined by $d^*(\ell)=-2$, $d^*(\ell^2)=1$, and $d^*(\ell^\alpha)=0$ for $\ell$ a prime and $\alpha\ge 3$ an integer. By \cite[(35)]{Ser97},
  \[\sum_{N_1\mid N} d^*(N/N_1) A_{\text{main}}(k, N_1, T'_{p^m}) =\frac{\psi(N)^\new(k-1)}{12}\cdot p^{-m/2}\cdot \delta_{m\text{ even}},\]
  where $\psi(N)^\new=\sum_{N_1\mid N} d^*(N/N_1) N_1\prod_{\ell\mid N_1}(1+1/\ell)$, and by \cite[Section 9]{MurtSinh09},
  \[F(x)=\frac{\psi(N)^\new(k-1)}{12d_{k,N,\chi}}\cdot \frac{2(p+1)}{\pi}\cdot \frac{\sqrt{1-x^2}}{p+2+1/p-4x^2}.\]
  By \cite[(44,\,46,\,48)]{Ser97}, we find as in \cite[(8)]{MurtSinh09} that for any $N_1\mid N$,
  \begin{eqnarray*}
    |A_{\text{ell}}(k,N_1,\chi,T'_{p^m})|&\le&\frac{4e}{\log{2}}\cdot 2^{\omega(N_1)}p^{3m/2}\log(4p^{m/2}),\\
    |A_{\text{hyp}}(k,N_1,\chi,T'_{p^m})-A_{\text{hyp}}(k,N_1,\chi,T'_{p^{m-2}})|&\le&\sqrt{N_1}\tau(N_1),\\
    |A_{\text{par}}(k,N_1,T'_{p^m})-A_{\text{par}}(k,N_1,T'_{p^{m-2}})|&\le&p^{m/2}.
  \end{eqnarray*}
  Moreover, we note that $|d^*(n)|\le 2^{\omega(n)}\le \tau(n)\ll_{\varepsilon} n^{\varepsilon}$ for all integers $n$ (see \cite[(52)]{Ser97}). Hence, this yields with \eqref{eq:ErdosTuran}
  \begin{align}    \label{eq:Clowerbound2}
    C_{k,N,\chi}(A)\geq &\frac{\psi(N)^\new(k-1)}{12d_{k,N,\chi}}\cdot \frac{4(p+1)}{\pi}\int_0^A \frac{\sqrt{1-x^2}}{p+2+1/p-4x^2}dx-\frac{1}{M+1}\\ \label{eq:Clowerbound}
                      &-c(\varepsilon ) N^\varepsilon\left(\frac{4e}{\log{2}}\cdot \frac{p^{3M/2}}{d_{k,N,\chi}}\cdot\log(4p^{M/2})-\frac{\sqrt{N}}{d_{k,N,\chi}}-\delta_{\substack{k=2\\\chi=1}} \cdot\frac{p^{M/2}}{d_{k,N,\chi}}\right),
  \end{align}
  for any $\varepsilon>0$, with $c(\varepsilon) > 0$ a constant depending only on $\varepsilon$.  As in \cite[(61, 62)]{Ser97},
  \[d_{k,N,\chi} =\frac{k-1}{12}\cdot \psi(N)^\new+O \left(N^{1/2+\varepsilon}\right),\]
  therefore given $\varepsilon<1/100$ positive,  as long as $M \leq (2/3 - 3\varepsilon) \log (N) / \log p$, all the three terms in \eqref{eq:Clowerbound} are less than $c'(\varepsilon) N^{-\varepsilon/100}$ for all $N$ and some constants $c'(\varepsilon)$ depending only on $\varepsilon$\footnote{The choice of $M$ is motivated by the fact that the growth of \eqref{eq:Clowerbound} is dominated by the first term in \eqref{eq:Clowerbound} which is roughly of size $N^{\varepsilon} p^{3 M / 2} N^{-1}$. Thus it is sufficient to choose $M$ so that this term is negligible, that is $p^{3 M / 2} N^{\varepsilon - 1} \ll N^{-\varepsilon}$.}.

  By a Taylor expansion at $x = 0$,
$$
\int_0^A \frac{\sqrt{1-x^2}}{p+2+1/p-4x^2}dx = \frac{p}{(p + 1)^2} \cdot A\cdot (1 + O(A)),
$$
therefore
\begin{eqnarray*}
  C_{k,N,\chi}(A)& \geq &\frac{4}{\pi}\cdot \frac{p}{p+1}\cdot A  (1+O(A))-\frac{1}{M+1}-\frac{c'(\varepsilon)}{N^{\varepsilon/100}}.
\end{eqnarray*}
  Hence, given $\varepsilon > 0$, choosing $A$ so that,
  $$
  \frac{4}{\pi} \cdot \frac{p}{p + 1} \cdot A > \Big ( \frac{3}{2} + \varepsilon \Big )\cdot \frac{\log p}{\log N} > \frac{1 + \varepsilon}{M + 1}
  $$
  ensures that $C_{k,N,\chi}(A) > 0$ for all sufficiently large $N$. In particular fixing a sufficiently small $\varepsilon > 0$ we see that for all $N$ large enough any 
  $$
  A > \frac{\pi}{2} \cdot \frac{p + 1}{p} \cdot \frac{\log p}{\log N}
  $$
  is acceptable.
  
\end{proof}

Theorem \ref{thm1} and Theorem \ref{thm2} now follows from combining Proposition \ref{lem1} and Lemma \ref{lem2} and specializing accordingly. 

\section{Proof of Theorem \ref{thm3}}

For $k\ge 2$ and $N\ge 1$ square-free, let $f\in S_k(\Gamma_0(N),\chi)$ be a newform. We factor the character $\chi$ as $\prod_{p \mid  N} \chi_{p}$ with $\chi_{p} : (\mathbb{Z} / p)^{\times} \rightarrow \mathbb{C}^{\times}$ a character modulo $p$.   The idea behind Theorem \ref{thm3} is inspired by \cite{ChoieKohn06}, where Choie and Kohnen show that the non-diagonalizability of a ``bad'' Hecke operator $T_p$ (i.e. with $p\mid N$) implies that $\sqrt{p}\in \Q(a_n(f): n\ge 1)$, and hence that this field has degree at least $2^s$ if $s$ such operators are non-diagonalizable.

Let
$$
N_2 = \prod_{\substack{p \mid N \\ \chi_{p} = 1}} p
$$
and write $N = N_1 N_2$, with $(N_1, N_2) = 1$ since $N$ is square-free. It follows that $\chi = \chi_{N_1} \chi_{N_2}$ with $\chi_{N_1}$ a primitive character of modulus $N_1$ and $\chi_{N_2} = 1$ the principal character modulo $N_2$.  Our argument is based on the Atkin--Lehner operators
\begin{align*}W_{p} & : S_k(\Gamma_0(N),\chi)\to S_k(\Gamma_0(N),\overline\chi_{p} \chi_{N / p}) \ , \ p \mid N 
\end{align*}
where $\chi_{N / p} = \prod_{\ell \mid N / p} \chi_{\ell}$ and on the properties of the pseudo-eigenvalues $\lambda_{p}(f)$ studied by Atkin and Li \cite{Li74,AtkinLi78}. Examining these elements gives bounds on the degrees of Fourier coefficients $a_f(p)$ at ``bad'' primes $p\mid N_2$. In turn, this yields lower bounds on $\deg{K_f}$ since:
\begin{lemma}\label{lemma:strongMult}
  We have $K_f=\Q(a_f(n) : n\ge 1)$.
\end{lemma}
\begin{proof}
  Let $K := \Q(a_f(n): n \geq 1)$ and let $L$ be its Galois closure. By the Hecke relations $a_f(p)^2 = a_f(p^2) - p^{k - 1}\chi(p)$ for all $p \nmid N$, we have the tower of extensions $\Q(\zeta_{\ord \chi})\subset K_f\subset K\subset L$. By Galois theory, it suffices to show that $\Gal(L/K_f)\subset\Gal(L/K)$. To that effect, let $\sigma\in\Gal(L/K_f)$. By the fact that $\chi^\sigma=\chi$ and \cite[Corollary 12.4.5]{DiamondIm95}, $f^\sigma$ is a newform in $S_k(\Gamma_0(N),\chi)$ whose Fourier coefficients coincide with those of $f$ at all integers co-prime to $N$. By strong multiplicity one \cite[Theorem 6.2.3]{DiamondIm95}, $f=f^\sigma$, so that $\sigma$ fixes all coefficients of $f$, i.e. $\sigma$ fixes $K$.
\end{proof}

Recall that for $p \mid  N$, the pseudo-eigenvalue $\lambda_{p}(f)\in\C$ is defined by the equation
\[W_{p}f=\lambda_{p}(f)g,\]
where $g\in S_k(\Gamma_0(N),\overline\chi_{p}\chi_{N / p})$ is a newform (see \cite[p.224]{AtkinLi78}) given by
\begin{equation}
  \label{eq:agp}
  a_g(\ell)=
  \begin{cases}
    \overline{\chi}_{p}(\ell) a_f(\ell)&:\ell \neq p\\
    \chi_{N/p}(p)\overline{a_f(p)}&: \ell = p
  \end{cases}
\end{equation}
for primes $\ell$ (\cite[(1.1)]{AtkinLi78}).

In general, we only know that the pseudo-eigenvalue $\lambda_{p}(f)$ is algebraic with modulus 1 (\cite[Theorem 1.1]{AtkinLi78}). However, under additional assumptions on $\chi$, we have the following information on its field of definition:
\begin{lemma}\label{lemma:fieldsDefinition} Let $p \mid  N_2$. Then,
  $\lambda_{p}(f)\in\Q(\zeta_{2\ord(\chi)}) $.
\end{lemma}
\begin{proof}
  From the identity $W_{p}^2=\chi_{p}(-1)\overline\chi_{N/p}(p)\id$ (\cite[Proposition 1.1]{AtkinLi78}), we get that
  \begin{equation}
    \label{eq:lambdafg}
    \lambda_{p}(f)\lambda_{p}(g)=\chi_{p}(-1)\overline\chi_{N/p}(p)=\pm\overline\chi_{N/p}(p).
  \end{equation}
  Since $p \mid  N_2$ we have $\chi_{p} = 1$, so that $g\in S_k(\Gamma_0(N),\chi)$, and $a_g(\ell)=a_f(\ell)$ for all prime $\ell \neq p$, by \eqref{eq:agp}. By strong multiplicity one, we get $g=f$. By \eqref{eq:lambdafg}, we obtain $\lambda_{p}(f)^2=\overline\chi_{N/p}(p)$ and thus the claim. 
\end{proof}
The next ingredient is the explicit determination of $\lambda_f(p)$ in terms of $a_f(p)$ by Atkin and Li.
\begin{lemma}\label{lemma:explicitEV}
  Let $p \mid  N_2$. Then $a_f(p) \neq 0$
 and
  \[\lambda_p(f)=-\frac{p^{k/2-1}}{a_f(p)}.\]
\end{lemma}
\begin{proof}
  The fact that $a_{f}(p)\neq 0$ is \cite[Theorem 3(ii)]{Li74}, and the formula for the eigenvalue is \cite[Theorem 2.1]{AtkinLi78}.
\end{proof}

\begin{proof}[Proof of Theorem \ref{thm3}]
  By Lemmas \ref{lemma:strongMult}, \ref{lemma:fieldsDefinition} and \ref{lemma:explicitEV}, we get
\begin{eqnarray*}
  \{p^{k/2} : p\mid N_2\}&\subset& K_f(\zeta_{2\ord(\chi)}).
\end{eqnarray*}
Since $L:=\Q(\zeta_{\ord(\chi)}) \subset K_f$, we have
\begin{eqnarray*}
  [K_f:\Q]&\ge& \frac{1}{2}\cdot [K_f(\zeta_{2 \ord(\chi)}): \Q]\\
  &=&\frac{1}{2}\cdot [K_f(\zeta_{2 \ord(\chi)}):L]\cdot \varphi(\ord(\chi)),
\end{eqnarray*}
where the last factor is the trivial bound.

The square roots of odd primes $p\mid \ord(\chi)$ belong to $L$. On the other hand, for $S := \{ \sqrt{p} : p \mid N_2,\, p \nmid 2 \ord(\chi)\} \subset K_f(\zeta_{2\ord(\chi)})$, we have
\[[K_f(\zeta_{2 \ord(\chi)}):L]\ge[L(S):L]=2^{|S|}\]
by \cite[Theorem 87]{Hil98}, and the claim follows.
\end{proof}
\begin{remark}
 Since the character $\chi_{p}$ is primitive for $p \mid  N_1$, \cite[Theorem 3(ii)]{Li74} and \cite[Theorem 2.1, Proposition 1.4]{AtkinLi78} show that $\lambda_{p}(f)=p^{k/2-1}g(\chi_{p})/a_{f}(p)$, with $g(\chi_{p})$ the Gauss sum attached to $\chi_{p}$. The degree of $p^{k/2-1}g(\chi_{p})$ over $\Q$ can be determined precisely, however we have no information about the field of definition of $\lambda_{p}(f)$, except the fact that it is a root of unity. If we could show that it belongs to a small extension of $K_f$, in the same way as we did for $\lambda_{p}(f)$ with $p \mid  N_2$, then we could add a factor as large as $\ord(\chi)$ to the lower bound of Theorem \ref{thm3}, including when $k$ is even.
\end{remark}
\bibliographystyle{alpha}
\bibliography{references3}

\end{document}